\documentclass[12pt,amstex]{amsart}
\usepackage{stmaryrd}
\usepackage{enumerate}
\usepackage{epsfig}
\usepackage{amsmath}
\usepackage{amssymb}
\usepackage{amscd}
\usepackage{graphicx}

\usepackage{pstricks}

\usepackage{tocvsec2}

\usepackage{hyperref}

\input xy
\xyoption{all}

\newtheorem{thm}{Theorem}[section]
\newtheorem{lem}[thm]{Lemma}

\newtheorem{prop}[thm]{Proposition}
\newtheorem{ques}[thm]{Question}
\newtheorem{cor}[thm]{Corollary}

\theoremstyle{definition}

\theoremstyle{remark}
\newtheorem{rem}[thm]{Remark}

\numberwithin{equation}{section}
\def \Z {\mathbb Z}

\def \N {\mathbb N}

\begin{document}
\title{measure-theoretic equicontinuity and rigidity}
\author{Fangzhou Cai}
\date{}
\begin{abstract}
Let $(X,T)$ be a topological dynamical system  and $\mu$ be a invariant measure, we show that $(X,\mathcal{B},\mu,T)$ is rigid if and only if there exists some subsequence $A$ of $\N$ such that $(X,T)$ is $\mu$-$A$-equicontinuous if and only if there exists some IP-set $A$ such that $(X,T)$ is $\mu$-$A$-equicontinuous. We  show that  if there exists some  subsequence $A$ of $\N$ with positive upper density such that  $(X,T)$ is $\mu$-$A$-mean-equicontinuous, then  $(X,\mathcal{B},\mu,T)$  is rigid. We also give results with respect to a function.
\end{abstract}

\maketitle

\section{introduction}
Throughout this paper, a topological dynamical system(t.d.s. for short) is a pair $(X,T)$, where $X$ is an non-empty compact metric space with a metric $d$ and $T$ is a homeomorphism  from $X$ to itself. For a t.d.s. $(X,T)$, denote by $M(X,T)$ the set of $T$-invariant regular Borel probability measures on $X$. It is well known that there exists some $\mu\in M(X,T)$, thus $(X,\mathcal{B},\mu,T)$ can be viewed as a measure preserving system(m.p.s. for short), where $\mathcal{B}$ is the Borel $\sigma$-algebra. 

A t.d.s. $(X,T)$ is called \textit{equicontinuous} if $\{T^n:n\in\N\}$ is  uniformly equicontinuous. For a t.d.s., equicontinuity represents predictability, one may see that the dynamical behaviour of such a system is very ``rigid". For example, it is well known that a transitive t.d.s. $(X,T)$  is equicontinuous if and only if it is topologically conjugate to a minimal rotation on a compact abelian metric group, if and only if it has topological discrete spectrum(see \cite{pw}).

The analogous concept of equicontinuity for a m.p.s. was also introduced. While studying cellular automata(a subclass of t.d.s.), Gilman \cite{Gil1,Gil2} introduced a notion of $\mu$-equicontinuity. Later Huang $et\ al$ \cite{hly} introduced a different definition of $\mu$-equicontinuity(which under some conditions are equivalent \cite{fe1}) and showed that $\mu$-equicontinuous systems have discrete spectrum. Garc\'ia-Ramos \cite{fe2} introduced a weakening of $\mu$-equicontinuity called $\mu$-mean-equicontinuity and showed that if a m.p.s. $(X,\mathcal{B},\mu,T)$ is an ergodic system, then it has discrete spectrum if and only if it is  $\mu$-mean-equicontinuous. Recently,  Huang $et\ al$ \cite{h} proved that for a general m.p.s. $(X,\mathcal{B},\mu,T)$, it has discrete spectrum if and only if it is  $\mu$-mean-equicontinuous.

The notion of rigidity in ergodic theory was first
introduced by Furstenberg and Weiss in \cite{ri}.  A m.p.s. $(X,\mathcal{B},\mu,T)$ is called \textit{rigid} if there exists a subsequence $\{t_k\}$ of $\N$ such that $T^{t_k}f\stackrel{L^2}\longrightarrow f$ for all $f\in L^2(\mu)$. So rigidity is a spectral property. A function $f\in L^2(\mu)$ is called \textit{rigid} if there exists a subsequence $\{t_k\}$ of $\N$ such that $T^{t_k}f\stackrel{L^2}\longrightarrow f$. Note that the sequence $\{t_k\}$ may depend on $f$. It is clear  that if $(X,\mathcal{B},\mu,T)$ is rigid, then every function in $L^2(\mu)$ is rigid. An interesting result showed that the converse is true (see \cite[Corollary 2.6]{ra}).

The notion of uniformly rigidity was introduced by
Glasner and Maon in \cite{ur} as a topological analogue of
rigidity in ergodic theory. A t.d.s. $(X,T)$ is \textit{uniformly rigid} if there exists a subsequence $\{t_k\}$ of $\N$ such that $T^{t_k}\rightarrow id$ uniformly on $X$, where $id$ is the identity mapping.
 
The concept of rigidity is close to  equicontinuity. 
In \cite{hy1}, the authors proved the following theorem:
\begin{thm}\cite[Lemma 4.1]{hy1}
Let $(X,T)$ be a t.d.s.. If $(X,T)$ is uniformly rigid, then there exists an IP-set $A$ such that $(X,T)$ is $A$-equicontinuous. If in addition $(X,T)$ is an E-system, the converse holds. 	
\end{thm}
In fact,  in this paper we prove that:
\begin{thm}
	Let $(X,T)$ be a t.d.s.. Then $(X,T)$ is uniformly rigid if and only if  there exists an IP-set $A$ such that $(X,T)$ is $A$-equicontinuous, if and only if  there exists a subsequence $A$ of $\N$ such that $(X,T)$ is $A$-equicontinuous.
\end{thm}

It is natural to consider the corresponding relation between   measure-theoretic rigidity and measure-theoretic equicontinuity. Following this idea, in our paper we prove that:  

\begin{thm}
Let $(X,T)$ be a t.d.s.   and $\mu\in M(X,T)$. Then $(X,\mathcal{B},\mu,T)$ is rigid if and only if there exists some subsequence $A$ of $\N$ such that $(X,T)$ is $\mu$-$A$-equicontinuous if and only if there exists some IP-set $A$ such that $(X,T)$ is $\mu$-$A$-equicontinuous.
\end{thm}
Inspired by the work in \cite{h}, we also show that:
\begin{thm}
Let $(X,T)$ be a t.d.s. and $\mu\in M(X,T)$. If there exists a subsequence $A$ of $\N$ with $\bar{D}(A)>0$ such that  $(X,T)$ is $\mu$-$A$-mean-equicontinuous, then  $(X,\mathcal{B},\mu,T)$ is rigid.
\end{thm}
Since rigidity implies zero entropy, we have following corollary:

\begin{cor}
	Let $(X,T)$ be a t.d.s.. If there exists a subsequence $A$ of $\N$ with $\bar{D}(A)>0$ such that $(X,T)$ is $A$-mean-equicontinuous, then the topological entropy of $(X,T)$ is zero.
\end{cor}

\section{preliminaries}
In this section we recall some notions and aspects of dynamical systems. 
\subsection{Subsets of non-negative integers}
In the article, integers and natural numbers are denoted by
$\Z$ and $\N$ respectively. 
Let $F$ be a set, denote by $|F|$ the number of elements of $F$. Let $A$ be a subset of $\N$. Define the \textit{lower density} $\underline{D}(A)$ and the \textit{upper density} $\bar{D}(A)$ of $A$ by:
\begin{equation*}
\underline{D}(A)=\liminf_{n\rightarrow\infty}\frac{|A\cap[0,n-1]|}{n},
\end{equation*}
\begin{equation*}
\bar{D}(A)=\limsup_{n\rightarrow\infty}\frac{|A\cap[0,n-1]|}{n}.
\end{equation*}
If $\underline{D}(A)=\bar{D}(A)$, denote by $D(A)=\underline{D}(A)=\bar{D}(A)$ the \textit{density} of $A$.

Let $\{b_i:i\in\N\}\subset\N$. Define
\begin{equation*}
 FS\{b_i\}=\{b_{i_1}+b_{i_2}+\cdots+b_{i_k}:i_1<i_2<\cdots<i_k,k\in\N\}.
\end{equation*}
We say $F\subset\N$ is an \textit{$IP$-set} if there exists a subset $\{b_i:i\in\N\}$ of $\N$ such that $F=FS\{b_i\}$.
\subsection{Discrete spectrum in measurable dynamics}

Let $(X,\mathcal{B},\mu,T)$ be a m.p.s., one can define the m.p.s. $(X\times X,\mathcal{B}\times\mathcal{B},\mu\times\mu,T)$ by $T(x,y)=(Tx,Ty)$ for $(x,y)\in X\times X$. One can define the Koopman operator  $U_T$ as an isometry of $L^p(\mu)$ into $L^p(\mu)$ by
$U_Tf=f\circ T$, where $1\leq p<\infty$. We say a function $f\in L^2(\mu)$ is \textit{almost periodic} if $\{U_T^nf:n\in\Z\}$ is pre-compact in $L^2(\mu)$, that is the closure of $\{U_T^nf:n\in\Z\}$ is compact in $L^2(\mu)$. We say that $(X,\mathcal{B},\mu,T)$ has \textit{discrete spectrum} if for all $f\in L^2(\mu)$, $f$ is almost periodic.

\subsection{Measure-theoretic equicontinuity}
Let $(X,T)$ be a t.d.s. and $\mu\in M(X,T)$, let $K$ be a subset of $X$ and $A=\{a_i:i\in\N\}$ be a subsequence of $\N$. 

	 We say that $K$ is \textit{$A$-equicontinuous} if for any $\epsilon>0$, there exists $\delta>0$ such that 
\begin{equation*}
x,y \in K,d(x,y)<\delta\Rightarrow d(T^{a_n}x,T^{a_n}y)<\epsilon,\forall n\in\N.
\end{equation*}

 We say $(X,T)$ is  \textit{$A$-equicontinuous} if $X$ is $A$-equicontinuous.

 We say $(X,T)$ is \textit{$\mu$-$A$-equicontinuous} if for any $\tau>0$, there exists a compact subset $K$ of $X$ with $\mu(K)>1-\tau$ such that $K$ is $A$-equicontinuous.

	By the regularity of $\mu$, the compactness of $K$ in the definition above can be ignored. It is clear that the definition of $\mu$-$A$-equicontinuity is independent on the choice of metric $d$.

\subsection{Sequence entropy for a measure}Let $(X,\mathcal{B},\mu,T)$ be a m.p.s. and $S=\{s_i:i\in\N\}$ be a subsequence of $\N$. Suppose $\xi$ and $\eta$ are two finite measurable partitions of $X$. Define  $\xi\vee\eta=\{A\cap B:A\in\xi,B\in\eta\}$. It is easy to see $\xi\vee\eta$ is also a finite partition of $X$. The \textit{entropy of $\xi$}, written $H_{\mu}(\xi)$, is defined by 
\begin{equation*}
H_{\mu}(\xi)=-\sum_{A\in\xi}\mu(A)\log\mu(A),
\end{equation*}
and the  \textit{entropy of $\xi$ given $\eta$}, written $H_{\mu}(\xi|\eta)$, is defined by
\begin{equation*}
H_{\mu}(\xi|\eta)=H_{\mu}(\xi\vee\eta)-H_{\mu}(\eta)=-\sum_{A\in\xi}\sum_{B\in\eta}\mu(A\cap B)\log\frac{\mu(A\cap B)}{\mu(B)}.
\end{equation*}
The \textit{sequence entropy of $\xi$ with respect to $(X,\mathcal{B},\mu,T)$ along $S$} is defined by
\begin{equation*}
h_{\mu}^S(T,\xi)=\limsup_{n\rightarrow\infty}\frac{1}{n}H_{\mu}(\bigvee_{i=1}^n T^{-s_i}\xi).
\end{equation*}
And the \textit{sequence entropy of $(X,\mathcal{B},\mu,T)$ along $S$} is
\begin{equation*}
h_{\mu}^S(T)=\sup_{\alpha}h_{\mu}^S(T,\alpha),
\end{equation*}
where supremum is taken over all finite measurable partitions. When $S=\N$, we simply write $h_{\mu}(T)$ and it is called the \textit{entropy of  $(X,\mathcal{B},\mu,T)$}.
\section{measure-theoretic  equicontinuity and rigidity  }
In this section, we study the relation between measure-theoretic  equicontinuity and rigidity. We prove that $(X,\mathcal{B},\mu,T)$ is rigid if and only if there exists some subsequence $A$ of $\N$ such that $(X,T)$ is $\mu$-$A$-equicontinuous. In the rest we give an IP-version of our result.
\subsection{Measure-theoretic equicontinuity, rigidity and sequence entropy}
In this subsection, we prove that  $\mu$-$A$-equicontinuity implies  rigidity via sequence entropy. 

\medskip

First we  give a characterization of measure-theoretic rigidity. It is worth pointing out the following theorem:
\begin{thm}\cite[Theorem 3.10]{hsy}
Let $(X,\mathcal{B},\mu,T)$ be a m.p.s.. Then $(X,\mathcal{B},\mu,T)$ is rigid if and only if there exists some IP-set $F$ such that $h_{\mu}^A(T)=0$ for any subsequence $A$ of $F$.
\end{thm}
\begin{lem}\cite[Lemma 3.3]{hsy}\label{lem1}
	Let $(X,\mathcal{B},\mu,T)$ be a m.p.s. and $B\in\mathcal{B}$. Let $A$ be a subsequence of $\N$. Then ${\{U_T^n1_B:n\in A\}}$ is pre-compact in $L^2(\mu)$ if and only if
	$h_{\mu}^S(T,\{B,B^c\})=0$ for any subsequence $S$ of $A$.
\end{lem}

Now we can give a characterization of rigidity:
\begin{thm}\label{th1}
Let $(X,\mathcal{B},\mu,T)$ be a m.p.s.. Then $(X,\mathcal{B},\mu,T)$ is rigid if and only if there exists some subsequence $A$ of $\N$ such that $h_{\mu}^S(T)=0$ for any subsequence $S$ of $A$, if and only if ${\{U_T^nf:n\in A\}}$ is pre-compact in $L^2(\mu)$ for all $f\in L^2(\mu)$.
\end{thm}
\begin{proof}
We only need to prove that: If there exists some subsequence $A$ of $\N$ such that $h_{\mu}^S(T)=0$ for any subsequence $S$ of $A$, then $(X,\mathcal{B},\mu,T)$ is rigid.
By Lemma \ref{lem1}, we have that for any $B\in\mathcal{B}$, $\overline{\{U_T^n1_B:n\in A\}}$ is compact in  $L^2(\mu)$. Let $\{f_j:j\in\N\}$ be a dense subset of $L^2(\mu)$, then for any $j\in\N$, $\overline{\{U_T^nf_j:n\in A\}}$ is compact in  $L^2(\mu)$. Hence
\begin{equation*}
\prod_{j\in\N}\overline{\{U_T^nf_j:n\in A\}} \text{ is compact in }  (L^2(\mu))^{\N}.
\end{equation*}
 It follows that
 \begin{equation*}
\{(U_T^nf_1,\cdots,U_T^nf_j,\cdots): n\in A\}\subset\prod_{j\in\N}\overline{\{U_T^nf_j:n\in A\}} 
\end{equation*}
is totally bounded.
For each $k\in\N$, there exist $n_1,\ldots,n_N\in A$ such that
\begin{equation*}
\{(U_T^nf_1,\cdots,U_T^nf_j,\cdots): n\in A\}\subset\bigcup_{i=1}^N B((U_T^{n_i}f_1,\cdots,U_T^{n_i}f_j,\cdots),\frac{1}{k}).
\end{equation*}
Without loss of generality, we assume $B((U_T^{n_1}f_1,\cdots,U_T^{n_1}f_j,\cdots),\frac{1}{k})$ contains infinite many elements of the left set. Hence there are infinite $n\in A$ such that
\begin{equation*}
\begin{split}
d((U_T^nf_1,\cdots,U_T^nf_j,\cdots),(U_T^{n_1}f_1,\cdots,U_T^{n_1}f_j,\cdots))&=\\
\sum_{j\in\N}\frac{\|U_T^nf_j-U_T^{n_1}f_j\|_{L^2}}{2^j}&=\sum_{j\in\N}\frac{\|U_T^{n-n_1}f_j-f_j\|_{L^2}}{2^j}<\frac{1}{k}.
\end{split}
\end{equation*}	
Take $t_k=n-n_1$, since $n$ can be sufficiently large, we can choose $\{t_k\}$ increasing. It follows that
$\|U_T^{t_k}f_j-f_j\|_{L^2}<\frac{2^j}{k}$  for all $j,k\in\N$.
Hence 
\begin{equation*}
U_T^{t_k}f_j\stackrel{L^2}\longrightarrow f_j\text{  as }k\rightarrow\infty\text{ for each }f_j.
\end{equation*}
It follows that
$U_T^{t_k}f\stackrel{L^2}\longrightarrow f$ for each $f\in L^2(\mu)$.
\end{proof}

Now we study the relation between measure-theoretic equicontinuity and  measure-theoretic sequence entropy.

Let $A=\{a_i:i\in\N\}$ be a subsequence of $\N$. Let $\mathcal{U}$ be a finite open cover of $X$ and $K$ be a subset of $X$. Denote by $N(\mathcal{U}|K)$  the minimum among the cardinalities of the subset of $\mathcal{U}$ which covers $K$, and set $C_A(\mathcal{U}|K)=\lim\limits_{n\rightarrow\infty}N(\bigvee\limits_{i=1}^{n}T^{-a_i}\mathcal{U}|K)$, where $\mathcal{U}\vee\mathcal{V}=\{U\cap V:U\in\mathcal{U},V\in\mathcal{V}\}$ for two finite open covers $\mathcal{U},\mathcal{V}$ of $X$. We recall that for two finite covers $\alpha,\beta$ of $X$, $\alpha\preceq\beta$ means that for each $B\in\beta$, there exists $A\in\alpha$ such that $B\subset A.$ 

\medskip

The following lemma is useful(see \cite[Lemma 2.3]{hyz}):
\begin{lem}\label{lem2}
Let $\alpha=\{A_1,A_2\ldots,A_k\}$ be a Borel partition of $X$ and $\epsilon>0$, then there exists a finite open cover $\mathcal{U}$ with $k$ elements such that for any $j\in\N$ and any Borel partition $\beta$ satisfying $T^{-j}\mathcal{U}\preceq\beta$, $H_\mu(T^{-j}\alpha|\beta)\leq\epsilon$.
\end{lem}

Now we prove that  $\mu$-$A$-equicontinuity implies rigidity. The idea is from \cite[Proposition 5.4]{hly}. 
\begin{thm}\label{th4}
Let $(X,T)$ be a t.d.s. and $\mu\in M(X,T)$. If there exists a subsequence $A$ of $\N$ such that $(X,T)$ is $\mu$-$A$-equicontinuous, then 	$h_{\mu}^S(T)=0$ for any subsequence $S$ of $A$, hence
 $(X,\mathcal{B},\mu,T)$ is rigid.
\end{thm}
\begin{proof}
\textbf{Claim}: For any $\tau>0$, there exists a compact subset $K$ of $X$ with $\mu(K)>1-\tau$ such that for any finite open cover $\mathcal{U}$, $C_A(\mathcal{U}|K)<\infty$.
	
Proof of Claim: Given $\tau>0$, since $(X,T)$ is $\mu$-$A$-equicontinuous, there exists a compact $A$-equicontinuous subset $K$ of $X$ with $\mu(K)>1-\tau$. For any finite open cover $\mathcal{U}$, let $\epsilon>0$ be a Lebesgue number of $\mathcal{U}$. Since $K$ is $A$-equicontinuous, there exists $\delta>0$ such that
\begin{equation*}
x,y\in K,d(x,y)<\delta\Rightarrow d(T^ax,T^ay)<\epsilon,\ \forall a\in A.
\end{equation*}
As $K$ is compact, there exist
$x_1,\ldots,x_L\in K$ such that
\begin{equation*}
 K=\bigcup_{i=1}^L(B(x_i,\frac{\delta}{2})\cap K).
\end{equation*}
 For any $a\in A$ and $1\leq i\leq L$,
we have that
\begin{equation*}
diam(T^a(B(x_i,\frac{\delta}{2})\cap K))\leq\epsilon.
\end{equation*}
Hence there exists $U_{a,i}\in\mathcal{U}$ such that
\begin{equation*}
 T^a(B(x_i,\frac{\delta}{2})\cap K)\subset U_{a,i}.
\end{equation*}
 Thus
 \begin{equation*}
 B(x_i,\frac{\delta}{2})\cap K\subset \bigcap_{a\in A}T^{-a}U_{a,i},\ \forall \ 1\leq i\leq L.
 \end{equation*}
  This implies $C_A(\mathcal{U}|K)\leq L<\infty$. The proof of Claim is complete.
	
	Now we prove $h_{\mu}^S(T)=0$ for any subsequence $S$ of $A$.
	Assume the contrary, there exists a subsequence $S$ of $A$ such that $h_{\mu}^S(T)>0$, then there exists a Borel partition $\alpha$  of $X$ and $\epsilon_0>0$ such that $h_{\mu}^S(T,\alpha)>\epsilon_0$.
	By Lemma \ref{lem2}, there exists a finite open cover $\mathcal{U}$ of $X$ such that for any $j\in\N$ and any Borel partition $\beta$ satisfying $T^{-j}\mathcal{U}\preceq\beta$, $H_\mu(T^{-j}\alpha|\beta)\leq\epsilon_0$.
	
	For any $\tau>0$, by Claim, there exists a compact set $K$ with $\mu(K)>1-\tau$ and $C_A(\mathcal{U}|K)<\infty$, set $C_A(\mathcal{U}|K)=L$. For any $n\in\N$, we pick out two sub-collections $\mathcal{E}^1_n,\mathcal{E}^2_n$ of $\bigvee\limits_{i=1}^{n}T^{-s_i}\mathcal{U}$ such that $\mathcal{E}^1_n$ is a cover of $K$ with $N(\bigvee\limits_{i=1}^{n}T^{-s_i}\mathcal{U}|K)$ elements, and $\mathcal{E}^2_n$ is a cover of
	$X$ with $N(\bigvee\limits_{i=1}^{n}T^{-s_i}\mathcal{U})$ elements. Enumerate them by
	\begin{equation*}
	\mathcal{E}^1_n=\{U_i\}_{i=1}^{m(n,1)},\mathcal{E}^2_n=\{V_i\}_{i=1}^{m(n,2)}.
	\end{equation*}
	Set
	\begin{equation*}
	\beta_1=\{K\cap(U_i-\bigcup_{k=1}^{i-1}U_k): i=1,\ldots,m(n,1)\},
	\end{equation*}
	\begin{equation*}
	\beta_2=\{K^c\cap(V_i-\bigcup_{k=1}^{i-1}V_k): i=1,\ldots,m(n,2)\}.
	\end{equation*}
	Then $\beta_1$ is a partition of $K$ with $|\beta_1|\leq N(\bigvee\limits_{i=1}^{n}T^{-s_i}\mathcal{U}|K)$ and
	$\beta_2$ is a partition of $K^c$ with $|\beta_2|\leq N(\bigvee\limits_{i=1}^{n}T^{-s_i}\mathcal{U})$. Put $\beta=\beta_1\cup\beta_2$, then $\beta$ is a partition of $X$ and $\beta\succeq \bigvee\limits_{i=1}^{n}T^{-s_i}\mathcal{U}$. Hence we have
	\begin{equation*}
	\begin{split}
	H_\mu(\beta)&=-\sum_{A\in\beta_1}\mu(A)\text{log}\mu(A)-\sum_{B\in\beta_2}\mu(B)\text{log}\mu(B)\\
	&  \leq(\mu(K)\text{log}|\beta_1|-\mu(K)\text{log}\mu(K))+ (\mu(K^c)\text{log}|\beta_2|-\mu(K^c)\text{log}\mu(K^c))\\
	&\leq \mu(K)\text{log}N(\bigvee_{i=1}^{n}T^{-s_i}\mathcal{U}|K)+
	\mu(K^c)\text{log}N(\bigvee_{i=1}^{n}T^{-s_i}\mathcal{U})+\text{log}2\\
	&\leq\mu(K)\text{log}L+n\mu(K^c)\text{log}N(\mathcal{U})+\text{log}2.
	\end{split}
	\end{equation*}
	On the other hand, since $\beta\succeq \bigvee\limits_{i=1}^{n}T^{-s_i}\mathcal{U}$, by the construction of $\mathcal{U}$,
	\begin{equation*}
	 H_\mu(T^{-s_i}\alpha|\beta)\leq\epsilon_0, \ \forall\ 1\leq i\leq n.
	\end{equation*}
	It follows that
	\begin{equation*}
	\begin{split}
	&H_\mu(\bigvee_{i=1}^{n}T^{-s_i}\alpha)\leq H_\mu((\bigvee_{i=1}^{n}T^{-s_i}\alpha)\vee\beta)
	\leq H_\mu(\beta)+H_\mu(\bigvee_{i=1}^{n}T^{-s_i}\alpha|\beta)\\ &\leq\mu(K)\text{log}L+n\mu(K^c)\text{log}N(\mathcal{U})+\text{log}2+\sum_{i=1}^{n} H_\mu(T^{-s_i}\alpha|\beta)\\
	&\leq\mu(K)\text{log}L+n\mu(K^c)\text{log}N(\mathcal{U})+\text{log}2+n\epsilon_0.
	\end{split}
	\end{equation*}
	Hence
	\begin{equation*}
	h_{\mu}^S(T,\alpha)=\limsup_{n\rightarrow\infty}\frac{1}{n} H_\mu(\bigvee_{i=1}^{n}T^{-s_i}\alpha)
	\leq \mu(K^c)\text{log}N(\mathcal{U})+\epsilon_0\leq \tau\text{log}N(\mathcal{U})+\epsilon_0.
	\end{equation*}
	Let $\tau\rightarrow 0,$ we have $h_{\mu}^S(T,\alpha)\leq\epsilon_0$, it is a contradiction. Hence $h_{\mu}^S(T)=0$ for any subsequence $S$ of $A$. By Theorem \ref{th1}, $(X,\mathcal{B},\mu,T)$ is rigid.
\end{proof}

\subsection{Measure-theoretic rigidity and equicontinuity  }In this subsection, we prove that if $(X,\mathcal{B},\mu,T)$ is rigid, then we can find  some subsequence $A$ of $\N$ such that $(X,T)$ is $\mu$-$A$-equicontinuous.

\medskip

First we give a useful characterization of measure-theoretic equicontinuity:
\begin{prop}\label{ppp}
	Let $(X,T)$ be a t.d.s. and  $\mu\in M(X,T)$. Let $A=\{a_i:i\in\N\}$ be a subsequence of $\N$. Then the following statements are equivalent:
	\begin{enumerate}
		\item $(X,T)$ is $\mu$-$A$-equicontinuous.
		\item For any $\tau>0$ and $\epsilon>0$, there exist a compact subset $K$ of $X$ with $\mu(K)>1-\tau$ and $\delta>0$ such that 
		\begin{equation*}
		x,y\in K,d(x,y)<\delta\Rightarrow d(T^{a_n}x,T^{a_n}y)<\epsilon,\ \forall n\in\N.
		\end{equation*}
	\end{enumerate}
\end{prop}
\begin{proof}
	(1)$\Rightarrow$(2) is obvious.
	
	(2)$\Rightarrow$(1): Given $\tau>0$, for any $l\in\N$, by (2) there exist a compact set $K_l$  with $\mu(K_l)>1-\frac{\tau}{2^l}$ and $\delta_l>0$ such that for any $x,y\in K_l$ with $d(x,y)<\delta_l$, we have $d(T^{a_n}x,T^{a_n}y)<\frac{1}{l}$ for all $n\in\N.$ Let $K=\bigcap\limits_{l=1}^\infty K_l$, then $\mu(K)>1-\tau$ and $K$ is compact.  It is easy to see that $K$ is $A$-equicontinuous, hence  $(X,T)$ is $\mu$-$A$-equicontinuous.
\end{proof}

Now we prove that rigidity implies $\mu$-$A$-equicontinuity. Note that the following theorem is stronger.
\begin{thm}\label{th3}
Let $(X,T)$ be a t.d.s. and $\mu\in M(X,T)$. If there exists a subsequence $A=\{a_i:i\in\N\}$ of $\N$ such that $U_T^{a_n}d\stackrel{a.e.}\longrightarrow d$, where $d$ is the metric on $X$, then $(X,T)$ is $\mu$-$A$-equicontinuous.
\end{thm}
\begin{proof}
	It is similar to Theorem \ref{fm}.
\end{proof}
Combining Theorem \ref{th4} we have:
\begin{thm}\label{main}
Let $(X,T)$ be a t.d.s.  and $\mu\in M(X,T)$. Then $(X,\mathcal{B},\mu,T)$ is rigid if and only if there exists a subsequence $A$ of $\N$ such that $(X,T)$ is $\mu$-$A$-equicontinuous.
\end{thm}

\subsection{Results with respect to a function}
In this subsection we give corresponding results with respect to a function. 

Following the idea in \cite{f1,f2}, we introduce the notion of $f$-$\mu$-$A$-equicontinuity. we prove that for a function $f\in L^2(\mu)$, $f$ is rigid if and only if there exists some subsequence $A$ of $\N$ such that $f$ is $\mu$-$A$-equicontinuous.
\medskip

Let $(X,T)$ be a t.d.s. and $\mu\in M(X,T)$, let $K$ be a subset of $X$, $f\in L^2(\mu)$ and $A=\{a_i:i\in\N\}$ be a subsequence of $\N$. 

We say that $K$ is \textit{$f$-$A$-equicontinuous}(or \textit{$f$ is $A$-equicontinuous on $K$}) if for any $\epsilon>0$, there exists $\delta>0$ such that 
\begin{equation*}
x,y \in K,d(x,y)<\delta\Rightarrow |f(T^{a_n}x)-f(T^{a_n}y)|<\epsilon,\forall n\in\N.
\end{equation*}


We say $f$ is \textit{$\mu$-$A$-equicontinuous} if for any $\tau>0$, there exists a compact subset $K$ of $X$ with $\mu(K)>1-\tau$ such that $K$ is $f$-$A$-equicontinuous.

\medskip

Similar as Proposition \ref{ppp}, we give a characterization of $f$-$\mu$-$A$-equicontinuity:
\begin{prop}\label{p}
	Let $(X,T)$ be a t.d.s., $\mu\in M(X,T)$ and $f\in L^2(\mu)$. Let $A=\{a_i:i\in\N\}$ be a subsequence of $\N$. Then the following statements are equivalent:
	\begin{enumerate}
		\item $f$ is $\mu$-$A$-equicontinuous.
		\item For any $\tau>0$ and $\epsilon>0$, there exist a compact subset $K$ of $X$ with $\mu(K)>1-\tau$ and $\delta>0$ such that 
		\begin{equation*}
		x,y\in K,d(x,y)<\delta\Rightarrow|U_T^{a_n}f(x)-U_T^{a_n}f(y)|<\epsilon,\ \forall n\in\N.
		\end{equation*}
	\end{enumerate}
\end{prop}

Now we prove that for a function $f\in L^2(\mu)$, $f$-$\mu$-$A$-equicontinuity implies  rigidity.
\begin{thm}\label{f}
	Let $(X,T)$ be a t.d.s., $\mu\in M(X,T)$ and $f\in L^2(\mu)$. If there exists a subsequence $A=\{a_i:i\in\N\}$ of $\N$ such that $f$ is $\mu$-$A$-equicontinuous, then $\{U_T^{a_n}f:n\in \N\}$ is pre-compact in $L^2(\mu)$, hence $f$ is rigid.
\end{thm}
\begin{proof}
	Assume that $\{U_T^{a_n}f:n\in \N\}$ is not pre-compact in $L^2(\mu)$, then
	there exists $\epsilon>0$ and a subsequence $\{n_i\}$ of $\N$ such that
	\begin{equation*}
	||U^{a_{n_i}}_Tf-U^{a_{n_j}}_Tf||^2_{L^2}>9\epsilon^2+4\epsilon, \ \forall i\neq j\in\N.
	\end{equation*}
	Since $f\in L^2(\mu)$, there exists $\eta>0$ such that
	\begin{equation*}
	\mu(E)<\eta\Rightarrow\int_E |f|^2d\mu<\epsilon,
	\ \forall E\subset X.
	\end{equation*}

	\textbf{Claim}: For any subsequence $\{s_i\}$ of $\N$, we have
	\begin{equation*}
	\mu(\{x:|f(T^{s_i}x)|\rightarrow+\infty\})=0.
	\end{equation*}
	Proof of Claim: Let $M=\{x:|f(T^{s_i}x)|\rightarrow+\infty\}$. It is easy to see that $M$ is measurable. By Fatou's lemma
	\begin{equation*}
	\int_M\liminf_{i\rightarrow\infty}|f(T^{s_i}x)|d\mu(x)\leq\liminf_{i\rightarrow\infty}\int_M|f(T^{s_i}x)|d\mu(x).
	\end{equation*}
	Note that 
	\begin{equation*}
	\int_M|f(T^{s_i}x)|d\mu(x)\leq\int|f(T^{s_i}x)|d\mu(x)=\int|f|d\mu<+\infty,
	\end{equation*}
	we have $+\infty \cdot \mu(M)<+\infty$, hence $\mu(M)=0$. The proof of Claim is complete.
	
	Since $f$ is $\mu$-$A$-equicontinuous, there exists a $f$-$A$-equicontinuous set $K_0$ in $X$ with $\mu(K_0)>1-\eta$, hence there exists $\delta>0$ such that
	\begin{equation*}
	x,y\in K_0,d(x,y)<\delta\Rightarrow|f(T^{a_i}x)-f(T^{a_i}y)|<\epsilon,\ \forall i\in\N.
	\end{equation*}
	By the compactness of $X$, there exist
	\begin{equation*}
	K_0=X_1\cup X_2\cup\cdots \cup X_{m_0},
	\end{equation*}
	where $diamX_i<\delta,\ 1\leq i\leq m_0$. We assume $\mu(X_i)>0,\
	1\leq i\leq m;\ \mu(X_i)=0,\ m+1\leq i\leq m_0$.
	Let $K=X_1\cup X_2\cup\cdots \cup X_m$, then $K\subset K_0$ and $\mu(K)=\mu(K_0)>1-\eta$. Since $\mu(X_1)>0$ and
	\begin{equation*}
	\mu(\{x:|f(T^{a_{n_i}}x)|\rightarrow+\infty\})=0,
	\end{equation*}
	we can find $y_1\in X_1$ and $|f(T^{a_{n_i}}y_1)|$  not convergent to $+\infty$, hence there exists a subsequence $\{a_{n_i}^1\}\subset \{a_{n_i}\}$ such that $\{f(T^{a_{n_i}^1}y_1)\}$ is a Cauchy sequence. Note that  $\mu(X_2)>0$ and
	\begin{equation*}
	\mu(\{x:|f(T^{a_{n_i}^1}x)|\rightarrow+\infty\})=0,
	\end{equation*}
	we can find $y_2\in X_2$ and a subsequence $\{a_{n_i}^2\}\subset \{a_{n_i}^1\}$ such that $\{f(T^{a_{n_i}^2}y_2)\}$ is a Cauchy sequence.
	Working inductively, we can find $y_i\in X_i,1\leq i\leq m$ and a subsequence $\{s_i\}\subset \{a_{n_i}\}$ such that $\{f(T^{s_i}y)\}$ is a Cauchy sequence for all $y\in\{y_1,y_2\ldots,y_m\}$. Hence there exists $N\in\N$ such that
	\begin{equation*}
	\forall i,j>N,\ |f(T^{s_i}y)-f(T^{s_j}y)|<\epsilon,\ \forall y\in\{y_1,y_2\ldots,y_m\}.
	\end{equation*}
	Choose $i_0\neq j_0>N$, now we estimate $||U_T^{s_{i_0}}f-U_T^{s_{j_0}}f||_{L^2}^2$.
	\begin{equation*}
	||U_T^{s_{i_0}}f-U_T^{s_{j_0}}f||_{L^2}^2=\int_K|f\circ T^{s_{i_0}}-f\circ T^{s_{j_0}}|^2d\mu+\int_{K^c}|f\circ T^{s_{i_0}}-f\circ T^{s_{j_0}}|^2d\mu.
	\end{equation*} 
	We first estimate $\int_K|f\circ T^{s_{i_0}}-f\circ T^{s_{j_0}}|^2d\mu$. If $x\in K$, then there exists $y\in\{y_1,y_2\ldots,y_m\}$ such that $d(x,y)<\delta$. Since $x,y\in K_0$, we have
	\begin{equation*}
	|f(T^{a_i}x)-f(T^{a_i}y)|<\epsilon,\ \forall i\in\N, 
	\end{equation*}
	hence
	\begin{equation*}
	|f(T^{s_i}x)-f(T^{s_i}y)|<\epsilon,\ \forall i\in\N. 
	\end{equation*}
	 Since 
	$|f(T^{s_{i_0}}y)-f(T^{s_{j_0}}y)|<\epsilon,$
	  it follows that 
	\begin{equation*} 
	\begin{split}  
	|f(T^{s_{i_0}}x)-f(T^{s_{j_0}}x)|&\leq |f(T^{s_{i_0}}x)-f(T^{s_{i_0}}y)|+ |f(T^{s_{j_0}}x)-f(T^{s_{j_0}}y)|\\&+|f(T^{s_{i_0}}y)-f(T^{s_{j_0}}y)|<3\epsilon. 
	\end{split}
	\end{equation*}
	Hence  $\int_K|f\circ T^{s_{i_0}}-f\circ T^{s_{j_0}}|^2d\mu<9\epsilon^2$.
	\\
	Now we estimate
	$\int_{K^c}|f\circ T^{s_{i_0}}-f\circ T^{s_{j_0}}|^2d\mu$.
	\begin{equation*}
	\begin{split}
	\int_{K^c}|f\circ T^{s_{i_0}}-f\circ T^{s_{j_0}}|^2d\mu&\leq\int_{K^c}2(|f\circ T^{s_{i_0}}|^2+|f\circ T^{s_{j_0}}|^2)d\mu\\&=2\int_{T^{s_{i_0}}K^c}|f|^2d\mu+2\int_{T^{s_{j_0}}K^c}|f|^2d\mu.
	\end{split}
	\end{equation*}
	Since 
	\begin{equation*}
	\mu(T^{s_{i_0}}K^c)=\mu(T^{s_{j_0}}K^c)=\mu(K^c)<\eta,
	\end{equation*}
	 the right side $<4\epsilon$. Hence 
	\begin{equation*}
	\int_{K^c}|f\circ T^{s_{i_0}}-f\circ T^{s_{j_0}}|^2d\mu<4\epsilon.
	\end{equation*}
	It follows that 
	\begin{equation*}
		||U_T^{s_{i_0}}f-U_T^{s_{j_0}}f||_{L^2}^2<9\epsilon^2+4\epsilon.
	\end{equation*}
	Note that $\{s_i\}\subset \{a_{n_i}\}$ and
	\begin{equation*} ||U^{a_{n_i}}_Tf-U^{a_{n_j}}_Tf||_{L^2}>9\epsilon^2+4\epsilon, \ \forall i\neq j\in\N,
	\end{equation*}
	it is a contradiction. So  $\{U_T^{a_n}f:n\in \N\}$ is pre-compact in $L^2(\mu)$.
	
	Now we proof $f$ is rigid. Since $\{U_T^{a_n}f:n\in \N\}$ is pre-compact in $L^2(\mu)$, for every $k\in\N$, there exists $N\in\N$ such that
	\begin{equation*}
	\{U_T^{a_n}f:n\in \N\}\subset\bigcup_{i=1}^NB(U_T^{a_i}f,\frac{1}{k}).
	\end{equation*}
	For $n\in\N$, there exists $1\leq i\leq N$ such that 
	\begin{equation*}
	||U_T^{a_n-a_i}f-f||_{L^2}=||U_T^{a_n}f-U_T^{a_i}f||_{L^2}<\frac{1}{k}.
	\end{equation*}
	Let $n_k=a_n-a_i$, since $a_n\rightarrow+\infty$ we can choose $\{n_k:k\in\N\}$ increasing, hence we have $||U_T^{n_k}f-f||_{L^2}<\frac{1}{k},\ U_T^{n_k}f\stackrel{L^2}\longrightarrow f$. 
\end{proof}


The next theorem shows that for a function $f\in L^2(\mu)$, rigidity implies $f$-$\mu$-$A$-equicontinuity.
\begin{thm}\label{a}
	Let $(X,T)$ be a t.d.s., $\mu\in M(X,T)$ and $f\in L^2(\mu)$. If there exists a subsequence $A=\{a_i:i\in\N\}$ of $\N$ such that $U_T^{a_n}f\stackrel{a.e.}\longrightarrow f$, then $f$ is $\mu$-$A$-equicontinuous.
\end{thm}
\begin{proof}
	$\forall\tau>0,\forall \epsilon>0$, let
	\begin{equation*}
	A_N=\{x:|U_T^{a_n}f(x)-f(x)|<\frac{\epsilon}{3},\ \forall n>N\}.
	\end{equation*}
	It is easy to see that $\{A_N\}$ is increasing and
	\begin{equation*}
	\{x:U_T^{a_n}f(x)\rightarrow f(x)\}\subset\bigcup_{N=1}^\infty A_N.
	\end{equation*}
	Hence $\mu(\bigcup\limits_{N=1}^\infty A_N)=1$ so there exists $N\in\N$ such that $\mu(A_N)>1-\dfrac{\tau}{2}$. Since $f,U_T^{a_1}f,\ldots,U_T^{a_N}f\in L^2(\mu)$, by Lusin's theorem, there exists a compact subset $K_0$ of $X$ with $\mu(K_0)>1-\dfrac{\tau}{2}$ such that  $f,U_T^{a_1}f,\ldots,U_T^{a_N}f$ are continuous hence uniformly continuous on $K_0$. Then there exists $\delta>0$ such that 
	\begin{equation*}
	x,y\in K_0,d(x,y)<\delta\Rightarrow |f(x)-f(y)|<\dfrac{\epsilon}{3},|U_T^{a_n}f(x)-U_T^{a_n}f(y)|<\epsilon,\ 1\leq n\leq N.
	\end{equation*}
	Choose a compact set $K\subset A_N\cap K_0$ with $\mu(K)>1-\tau$. For any $x,y\in K$ and $d(x,y)<\delta$, if $1\leq n\leq N$, then $|U_T^{a_n}f(x)-U_T^{a_n}f(y)|<\epsilon$; if $n>N$, then
	\begin{equation*}
	\begin{split}
	|U_T^{a_n}f(x)-U_T^{a_n}f(y)|\leq|U_T^{a_n}f(x)-f(x&)|+|U_T^{a_n}f(y)-f(y)|+|f(x)-f(y)|\\<\frac{\epsilon}{3}+\frac{\epsilon}{3}+\frac{\epsilon}{3}=\epsilon.
	\end{split}
	\end{equation*}
	Hence $|U_T^{a_n}f(x)-U_T^{a_n}f(y)|<\epsilon$ for all $n\in\N$. By Proposition \ref{p}, $f$ is $\mu$-$A$-equicontinuous.
\end{proof}

\begin{cor}
	Let $(X,T)$ be a t.d.s., $\mu\in M(X,T)$ and $f\in L^2(\mu)$. Then $f$ is rigid if and only if there exists a subsequence $A=\{a_i:i\in\N\}$ of $\N$ such that $f$ is $\mu$-$A$-equicontinuous.
\end{cor}

Using the notion of $f$-$\mu$-$A$-equicontinuity, we  give a different proof of Theorem \ref{main}:
\begin{proof}[alternative proof of Theorem \ref{main}]
	$(\Leftarrow)$: Assume $(X,T)$ is $\mu$-$A$-equicontinuous.
	For any $f\in C(X)$, it is clear that $f$ is $\mu$-$A$-equicontinuous. By Theorem \ref{f},
	\begin{equation*}
	\{U_T^{a_n}f:n\in \N\} \text{ is pre-compact in } L^2(\mu).
	\end{equation*}
	  Since  $C(X)$ is dense in $L^2(\mu)$, we have
	  \begin{equation*}
	  \{U_T^{a_n}f:n\in \N\} \text{ is pre-compact in } L^2(\mu), \forall f\in L^2(\mu).
	  \end{equation*}
	    By Theorem \ref{th1},
	$(X,\mathcal{B},\mu,T)$ is rigid.
	
	\medskip
	
$(\Rightarrow)$:
	Assume $(X,\mathcal{B},\mu,T)$ is rigid. Then there exists a subsequence $\{n_i\}$ of $\N$ such that
	\begin{equation*}
	U_T^{n_i}f\stackrel{L^2}\longrightarrow f, \forall f\in L^2(\mu).
	\end{equation*}  Let $\{f_n:n\in\N\}$ be a dense subset of $C(X)$. It is well known that 
	\begin{equation*}
	d(x,y)=\sum_{n=1}^\infty\dfrac{|f_n(x)-f_n(y)|}{2^{n+1}||f_n||_\infty}
	\end{equation*}
	is a compatible metric of $X$. Without loss of generality, assume $d$ is the metric on $X$. Since $U_T^{n_i}f_1\stackrel{L^2}\longrightarrow f_1$, we can find a subsequence $\{n_i^1\}$ of $\{n_i\}$ such that $U_T^{n_i^1}f_1\stackrel{a.e.}\longrightarrow f_1$. Note that $U_T^{n_i}f_2\stackrel{L^2}\longrightarrow f_2$, we have $U_T^{n_i^1}f_2\stackrel{L^2}\longrightarrow f_2$, so we can find a subsequence $\{n_i^2\}$ of $\{n_i^1\}$ such that $U_T^{n_i^2}f_2\stackrel{a.e.}\longrightarrow f_2$. Working inductively and by diagonal procedure, we can find a subsequence $A=\{a_i\}$ of $\{n_i\}$ such that $U_T^{a_i}f_n\stackrel{a.e.}\longrightarrow f_n$ for all $n\in\N$. By Theorem \ref{a}, $f_n$ is $\mu$-$A$-equicontinuous for all $n\in\N$. 
	
	Now we proof $(X,T)$ is $\mu$-$A$-equicontinuous. For any $\tau>0$ and $\epsilon>0$, there exists $N\in\N$ such that $\sum\limits_{n=N+1}^\infty\dfrac{1}{2^n}<\dfrac{\epsilon}{2}$. Since $f_1,\ldots,f_N$ are  $\mu$-$A$-equicontinuous, there exists a compact set
	$K$ with $\mu(K)>1-\tau$ such that $K$ is $f_i$-$A$-equicontinuous, $1\leq i\leq N$. Hence there exists $\delta>0$ such that  for any $x,y\in K$ and $d(x,y)<\delta$, we have 
	\begin{equation*}
	|f_n(T^{a_m}x)-f_nT^{a_m}(y)|<\dfrac{\epsilon}{2N}\cdot2^{n+1}||f_n||_\infty, \ \forall m\in\N,\ 1\leq n\leq N.
	\end{equation*}
	For any $m\in\N$, we have
	\begin{equation*}
	d(T^{a_m}x,T^{a_m}y)=\sum_{n=1}^\infty\dfrac{|f_n(T^{a_m}x)-f_n(T^{a_m}y)|}{2^{n+1}||f_n||_\infty}
	=\sum_{n=N+1}^\infty+\sum_{n=1}^N<\dfrac{\epsilon}{2}+\dfrac{\epsilon}{2}=\epsilon.
	\end{equation*}
	By Proposition \ref{ppp}, $(X,T)$ is $\mu$-$A$-equicontinuous.
\end{proof}

Now we  give the IP-version of our main result:
\begin{thm}
	Let $(X,T)$ be a t.d.s. and $\mu\in M(X,T)$. Then $(X,\mathcal{B},\mu,T)$ is rigid if and only if there exists an IP-set $F$ such that $(X,T)$ is $\mu$-$F$-equicontinuous. 
\end{thm}
\begin{proof}
	By Theorem \ref{main}, we only need to prove $(\Rightarrow)$.
	From the proof above, there exists a subsequence $A=\{a_i\}$ of $\N$ such that $U_T^{a_i}f_n\stackrel{a.e.}\longrightarrow f_n$ for all $n\in\N$.
	Hence
	\begin{equation*}
	d(T^{a_i}x,x)=\sum_{n=1}^\infty\dfrac{|f_n(T^{a_i}x)-f_n(x)|}{2^{n+1}||f_n||_\infty}\rightarrow 0 \text{ a.e.}.
	\end{equation*} 
	Since $T^{a_n}\stackrel{a.e.}\longrightarrow id$,  for each $i\in\N$, we can find $n_i\in\N$ such that 
	\begin{equation*}
	\mu(A_i)>1-\dfrac{1}{4^i},\text{ where }	A_i=\{x\in X:d(T^{a_{n_i}}x,x)<\dfrac{1}{2^i}\}.
	\end{equation*}
Let $F=FS\{a_{n_i}\}$, we prove that $(X,T)$ is $\mu$-$F$-equicontinuous.
For any $\tau>0$ and $\epsilon>0$, there exists $N\in\N$ such that $\dfrac{1}{2^{N-1}}<\min\{\tau,\dfrac{\epsilon}{3}\}$. Note that 
\begin{equation*}
\mu(\bigcap_{1\leq i_1<\cdots<i_k<i}T^{-(a_{n_{i_1}}+\cdots a_{n_{i_k}})}A_i)>1-\dfrac{1}{2^i},
\end{equation*}
we have 
\begin{equation*}
\mu(\bigcap_{i=N}^{\infty}\bigcap_{1\leq i_1<\cdots<i_k<i}T^{-(a_{n_{i_1}}+\cdots a_{n_{i_k}})}A_i)>1-\tau.
\end{equation*}
Let $K=\bigcap\limits_{i=N}^{\infty}\bigcap\limits_{1\leq i_1<\cdots<i_k<i}T^{-(a_{n_{i_1}}+\cdots a_{n_{i_k}})}A_i$, then $\mu(K)>1-\tau$. Since
\begin{equation*}
\{T^{a_{n_{i_1}}+\cdots a_{n_{i_k}}}:1\leq i_1<\cdots<i_k<N\}
\end{equation*}
  are uniformly equicontinuous, there exists $\delta>0$ such that 
\begin{equation*}
d(x,y)<\delta\Rightarrow d(T^{a_{n_{i_1}}+\cdots a_{n_{i_k}}}x,T^{a_{n_{i_1}}+\cdots a_{n_{i_k}}}y)<\dfrac{\epsilon}{3},1\leq i_1<\cdots<i_k<N.
\end{equation*}
For $x,y\in K, d(x,y)<\delta$, we prove that
$d(T^ax,T^ay)<\epsilon$ for all $a\in F$.
Let
\begin{equation*}
 a=a_{n_{i_1}}+\cdots a_{n_{i_k}}+a_{n_{i_{k+1}}}+\cdots a_{n_{i_t}},
\end{equation*}
 where $1\leq i_1<\cdots<i_k<N\leq i_{k+1}<\cdots<i_t$. Since $i_t\geq N$, we have 
 \begin{equation*}
 x\in\bigcap_{1\leq i_1<\cdots<i_k<i_t}T^{-(a_{n_{i_1}}+\cdots a_{n_{i_k}})}A_{i_t},
 \end{equation*}
  hence $T^{a_{n_{i_1}}+\cdots a_{n_{i_{t-1}}}}x\in A_{i_t}$. It follows that 
\begin{equation*}
d(T^{a_{n_{i_1}}+\cdots a_{n_{i_t}}}x,T^{a_{n_{i_1}}+\cdots a_{n_{i_{t-1}}}}x)<\dfrac{1}{2^{i_t}}.
\end{equation*}
Similarly
\begin{equation*}
d(T^{a_{n_{i_1}}+\cdots a_{n_{i_{t-1}}}}x,T^{a_{n_{i_1}}+\cdots a_{n_{i_{t-2}}}}x)<\dfrac{1}{2^{i_{t-1}}},
\end{equation*}
\begin{equation*}
\vdots
\end{equation*}
\begin{equation*}
d(T^{a_{n_{i_1}}+\cdots a_{n_{i_{k+1}}}}x,T^{a_{n_{i_1}}+\cdots a_{n_{i_{k}}}}x)<\dfrac{1}{2^{i_{k+1}}}.
\end{equation*}
Hence 
\begin{equation*}
d(T^{a_{n_{i_1}}+\cdots a_{n_{i_t}}}x,T^{a_{n_{i_1}}+\cdots a_{n_{i_{k}}}}x)<\dfrac{1}{2^{i_t}}+\dfrac{1}{2^{i_{t-1}}}+\cdots+\dfrac{1}{2^{i_{k+1}}}\leq\dfrac{1}{2^{N-1}}<\dfrac{\epsilon}{3}.
\end{equation*}
Similarly
\begin{equation*}
d(T^{a_{n_{i_1}}+\cdots a_{n_{i_t}}}y,T^{a_{n_{i_1}}+\cdots a_{n_{i_{k}}}}y)<\dfrac{\epsilon}{3}.
\end{equation*}
Then 
\begin{equation*}
\begin{split}
d(T^ax,T^ay)&\leq d(T^{a_{n_{i_1}}+\cdots a_{n_{i_t}}}x,T^{a_{n_{i_1}}+\cdots a_{n_{i_{k}}}}x)+d(T^{a_{n_{i_1}}+\cdots a_{n_{i_t}}}y,T^{a_{n_{i_1}}  +\cdots a_{n_{i_{k}}}}y)\\&+d(T^{a_{n_{i_1}}+\cdots a_{n_{i_k}}}x,T^{a_{n_{i_1}}+\cdots a_{n_{i_{k}}}}y)<\dfrac{\epsilon}{3}+\dfrac{\epsilon}{3}+
\dfrac{\epsilon}{3}=\epsilon.
\end{split}
\end{equation*}
By Proposition \ref{ppp}, $(X,T)$ is $\mu$-$F$-equicontinuous.
\end{proof}
Similarly we have: 
\begin{thm}
	Let $(X,T)$ be a t.d.s., $\mu\in M(X,T)$ and $f\in L^2(\mu)$. Then $f$ is rigid if and only if there exists an IP-set $F$ such that $f$ is $\mu$-$F$-equicontinuous. 
\end{thm}
\subsection{The topological case}
In \cite{hy1}, the authors proved the following theorem:
\begin{thm}\cite[Lemma 4.1]{hy1}
	Let $(X,T)$ be a t.d.s.. If $(X,T)$ is uniformly rigid, then there exists an IP-set $A$ such that $(X,T)$ is $A$-equicontinuous. If in addition $(X,T)$ is an E-system, the converse holds. 	
\end{thm}
In fact we have:
\begin{thm}
	Let $(X,T)$ be a t.d.s.. Then $(X,T)$ is uniformly rigid if and only if  there exists an IP-set $A$ such that $(X,T)$ is $A$-equicontinuous, if and only if  there exists a subsequence $A$ of $\N$ such that $(X,T)$ is $A$-equicontinuous.
\end{thm}
\begin{proof}
	We only need to prove that: If there exists a subsequence $A$ of $\N$ such that $(X,T)$ is $A$-equicontinuous, then $(X,T)$ is uniformly rigid.
	
	Assume there exists a subsequence $A=\{a_i:i\in\N\}$ of $\N$ such that $(X,T)$ is $\mu$-$A$-equicontinuous. By Ascoli's Theorem, $\{T^{a_n}:n\in\N\}$ is pre-compact in $C(X,X)$ with uniform topology.
	For every $k\in\N$, there exists $N\in\N$ such that
	\begin{equation*}
	\{T^{a_n}:n\in \N\}\subset\bigcup_{i=1}^NB(T^{a_i},\frac{1}{k}).
	\end{equation*}
	For $n\in\N$, there exists $1\leq i\leq N$ such that 
	\begin{equation*}
	||T^{a_n-a_i}-id||_\infty=||T^{a_n}-T^{a_i}||_\infty<\frac{1}{k}.
	\end{equation*}
	Let $n_k=a_n-a_i$, since $a_n\rightarrow+\infty$ we can choose $\{n_k:k\in\N\}$ increasing, hence we have $||T^{n_k}-id||_\infty<\frac{1}{k},\ T^{n_k}\rightarrow id$ uniformly.
\end{proof}
\section{measure-theoretic mean equicontinuity and rigidity}
In \cite{fe2}, the author introduced the concept of measure-theoretic mean equicontinuity and defined a notion called \textit{$\mu$-mean-equicontinuity}. In \cite{h}, the authors introduced a similar notion called \textit{$\mu$-equicontinuity in the mean} and showed that they are equivalent. In fact, it is shown in \cite{h} that they are equivalent to discrete spectrum for a general $\mu\in M(X,T)$.

Following the idea, in this section we consider the  measure-theoretic mean equicontinuity along some subsequence $A$ of $\N$.
 We  study the relation between measure-theoretic mean equicontinuity and rigidity.

\medskip

Let $(X,T)$ be a t.d.s. and $\mu\in M(X,T)$, let $K$ be a subset of $X$ and $A=\{a_i:i\in\N\}$ be a subsequence of $\N$. 

For $x,y\in X$, define

\begin{equation*}
\bar{d}_n^A(x,y)=\frac{1}{n}\sum_{i=0}^{n-1}d(T^{a_i}x,T^{a_i}y).  
\end{equation*}
We say that $K$ is \textit{$A$-equicontinuous in the mean} if for any $\epsilon>0$, there exists $\delta>0$ such that 
\begin{equation*}
x,y \in K,d(x,y)<\delta\Rightarrow \bar{d}_n^A(x,y)<\epsilon,\forall n\in\N.
\end{equation*}

We say $(X,T)$ is  \textit{$A$-equicontinuous in the mean} if $X$ is $A$-equicontinuous in the mean.

We say $(X,T)$ is \textit{$\mu$-$A$-equicontinuous in the mean} if for any $\tau>0$, there exists a compact subset $K$ of $X$ with $\mu(K)>1-\tau$ such that $K$ is $A$-equicontinuous in the mean.

Similar as Propositin \ref{ppp}, we have:
\begin{prop}\label{pp}
	Let $(X,T)$ be a t.d.s. and  $\mu\in M(X,T)$. Let $A=\{a_i:i\in\N\}$ be a subsequence of $\N$. Then the following statements are equivalent:
	\begin{enumerate}
		\item $(X,T)$ is $\mu$-$A$-equicontinuous in the mean.
		\item For any $\tau>0$ and $\epsilon>0$, there exist a compact subset $K$ of $X$ with $\mu(K)>1-\tau$ and $\delta>0$ such that 
		\begin{equation*}
		x,y\in K,d(x,y)<\delta\Rightarrow \bar{d}_n^A(x,y)<\epsilon,\ \forall n\in\N.
		\end{equation*}
	\end{enumerate}
\end{prop}

We say that $K$ is \textit{$A$-mean-equicontinuous} if for any $\epsilon>0$, there exists $\delta>0$ such that 
\begin{equation*}
x,y \in K,d(x,y)<\delta\Rightarrow \limsup_{n\rightarrow\infty}\bar{d}_n^A(x,y)<\epsilon.
\end{equation*}

 We say $(X,T)$ is  \textit{$A$-mean-equicontinuous} if $X$ is $A$-mean-equicontinuous.

We say $(X,T)$ is \textit{$\mu$-$A$-mean-equicontinuous } if for any $\tau>0$, there exists a compact subset $K$ of $X$ with $\mu(K)>1-\tau$ such that $K$ is $A$-mean-equicontinuous.
 
 \medskip
 
The notions of $\mu$-$A$-equicontinuous in the mean  and $\mu$-$A$-mean-equicontinuous are equivalent:
\begin{thm}
	Let $(X,T)$ be a t.d.s. and $\mu\in M(X,T)$. Let $A=\{a_i:i\in\N\}$ be a subsequence  of $\N$, then the following statements are equivalent:
	\begin{enumerate}
		\item $(X,T)$ is $\mu$-$A$-equicontinuous in the mean.
		\item $(X,T)$ is $\mu$-$A$-mean-equicontinuous.
	\end{enumerate}	
\end{thm}
\begin{proof}
	The proof is similar to Theorem \ref{eq}.
	\end{proof}
\begin{rem}
As the same in \cite{li}, it is easy to see that the definition of $\mu$-$A$-mean-equicontinuity is independent on the choice of metric $d$.
\end{rem}
\begin{rem}
	The property of mean-equicontinuity in a topological dynamical system has also been studied, we refer readers to \cite{me1,li,me2}. 
\end{rem}

Now we prove the main theorem of this section. The idea is from \cite[Theorem 2.21]{metric}.
\begin{thm}\label{mean}
Let $(X,T)$ be a t.d.s. and $\mu\in M(X,T)$. If there exists a subsequence $A=\{a_i:i\in\N\}$ of $\N$ with $\bar{D}(A)>0$ such that $(X,T)$ is $\mu$-$A$-mean-equicontinuous, then $(X,\mathcal{B},\mu,T)$ is rigid.
\end{thm}
\begin{proof}
Without loss of generality, we assume the metric $d\leq1$ and $\bar{D}(A)>\frac{1}{M},\ M\in\N$. By Theorem \ref{th3} and Theorem \ref{main}, we only need to prove  ${\{U_T^{a_n}d:n\in \N\}}$ is pre-compact in $L^1(\mu\times\mu)$.
Assume the contrary, then there exists $\epsilon>0$ and a subsequence $\{n_i\}$ of $\N$ such that
\begin{equation*}
\|U_T^{a_{n_i}}d-U_T^{a_{n_j}}d\|_{L^1}>17M\epsilon\text{ for any }i\neq j\in\N.
\end{equation*}
 Since $(X,T)$ is $\mu$-$A$-equicontinuous in the mean, there exist a compact set $K$ with $\mu(K)>1-\epsilon$ and $\delta>0$ such that
 \begin{equation*}
  x,y\in K, d(x,y)<\delta\Rightarrow \bar{d}_n^A(x,y)<\frac{\epsilon}{2},\ \forall n\in\N.
 \end{equation*}
By the compactness of $X$, there exist $K=\bigcup\limits_{i=1}^NX_i$, where $diamX_i<\delta,1\leq i\leq N$,\ $X_i\cap X_j=\emptyset, 1\leq i\neq j\leq N$. Without loss of generality, we assume $X_i\neq\emptyset$ for all $1\leq i\leq N$  and choose $x_i\in X_i,1\leq i\leq N$.

Let $C=(\frac{1}{[\epsilon]}+1)^{N^2}$. Choose $n_0\in\N$ sufficiently large such that
\begin{equation*}
 n_0>a_{n_{10MC}}\text{ and }|A\cap[1,n_0]|>\frac{n_0}{M}.
\end{equation*} 
Note that $diam_{\bar{d}_{n_0}^A}X_i\leq\epsilon,1\leq i\leq N$.  Let $X_0=K^c,$ then   $\{X_0,X_1,\ldots,X_N\}$ is a partition of $X$ with
\begin{equation*}
 \mu(X_0)<\epsilon,
\ diam_{\bar{d}_{n_0}^A}X_i\leq\epsilon,1\leq i\leq N.
\end{equation*} 
 
For any $1\leq s\leq n_0$, define
$$ d_s(x,y)=
\left\{
             \begin{array}{ll}
              0, & (x,y)\in X_0\times X\cup X\times X_0\hbox{;} \\
               U^{a_s}_Td(x_i,x_j), &(x,y)\in X_i\times X_j \ 1\leq i,j\leq N\hbox{.}
             \end{array}
           \right.
$$
For any $(x,y)\in X_i\times X_j$, we have 
\begin{equation*}
|d_s(x,y)-U^{a_s}_Td(x,y)|=|U^{a_s}_Td(x,y)-U^{a_s}_Td(x_i,x_j)|\leq U^{a_s}_Td(x,x_i)+U^{a_s}_Td(x_j,y).
\end{equation*}
Sum $s$ from $1$ to $n_0$, we have
\begin{equation*}
\begin{split}
\sum_{s=1}^{n_0}|d_s(x,y)-U^{a_s}_Td(x,y)|&\leq 
\sum_{s=1}^{n_0}U^{a_s}_Td(x,x_i)+\sum_{s=1}^{n_0}U^{a_s}_Td(x_j,y) \\
&=n_0\bar{d}_{n_0}^A(x,x_i)+n_0\bar{d}_{n_0}^A(x_j,y).
\end{split}
\end{equation*}
Hence
\begin{equation*}
 \sum_{s=1}^{n_0}|d_s(x,y)-U^{a_s}_Td(x,y)|\leq 2n_0\epsilon
\end{equation*}
since  $diam_{\bar{d}_{n_0}^A}X_i\leq\epsilon,1\leq i\leq N.$
Integrate on $X_i\times X_j$, we have
\begin{equation*}
\int_{X_i\times X_j}\sum_{s=1}^{n_0}|d_s-U^{a_s}_Td|\leq 2n_0\epsilon(\mu\times\mu)(X_i\times X_j).
\end{equation*}
 Note that 
 \begin{equation*}
\begin{split}
 \int_{X_0\times X\cup X\times X_0}\sum_{s=1}^{n_0}|d_s-U^{a_s}_Td|&=\int_{X_0\times X\cup X\times X_0}\sum_{s=1}^{n_0}|U^{a_s}_Td|\\
 &\leq n_0(\mu\times\mu)(X_0\times X\cup X\times X_0)\leq2n_0\epsilon,
 \end{split}
\end{equation*}
 we have
\begin{equation*}
 \sum_{s=1}^{n_0}||d_s-U^{a_s}_Td||_{L^1}\leq4n_0\epsilon.
\end{equation*}
It follows that
\begin{equation*}
 |\{1\leq s\leq n_0:||d_s-U^{a_s}_Td||_{L^1}\leq8M\epsilon\}|\geq n_0-[\frac{n_0}{2M}].
\end{equation*}
Note that
\begin{equation*}
\begin{split}
 |A\cap[1,n_0]|=|\{s\in\N: a_s\leq n_0\}|&=\\
 |\{1\leq s\leq n_0: a_s\leq n_0\}|&>\frac{n_0}{M}\geq2[\frac{n_0}{2M}],
\end{split}
\end{equation*}
we have
\begin{equation*}
|\{1\leq s\leq n_0:||d_s-U^{a_s}_Td||_{L^1}\leq8M\epsilon \text{ and }a_s\leq n_0\}|\geq[\frac{n_0}{2M}].
\end{equation*}
Denote the left set by $S$. Divide $[0,1]$ into $\frac{1}{[\epsilon]}+1$  intervals averagely, then $[0,1]^{N^2}$ is divided into $(\frac{1}{[\epsilon]}+1)^{N^2}=C$ cubes. Since
\begin{equation*}
 d_s(x,y)=\sum_{1\leq i,j\leq N}b_{ij}1_{X_i\times X_j}, 0\leq b_{ij}\leq 1,
\end{equation*}
$\{b_{ij}\}_{1\leq i,j\leq N}$ must be in one of the $C$ cubes. By pigeonhole principle, there exists $S_0\subset S$ with $|S_0|\geq[\frac{[\frac{n_0}{2M}]}{C}]\geq\frac{n_0}{4MC}$
such that
 \begin{equation*}
 ||d_s-d_{s'}||_{L^1}<\epsilon, \ \forall s\neq s' \in S_0.
  \end{equation*}
It follows that for any $s\neq s' \in S_0$,
\begin{equation*}
\begin{split}
||U^{a_s}_Td-U^{a_{s'}}_Td||_{L^1}&\leq||d_s-U^{a_s}_Td||_{L^1}+||d_{s'}-U^{a_{s'}}_Td||_{L^1}+||d_s-d_{s'}||_{L^1}\\
&\leq8M\epsilon+8M\epsilon+\epsilon\leq17M\epsilon.
\end{split}
\end{equation*} 
Enumerate $S_0$ as
$1\leq s_1<s_2<\cdots<s_k\leq n_0,\ k\geq\frac{n_0}{4MC}.$
Consider the set
\begin{equation*}
 \{a_{s_i}+a_{n_j}:1\leq i\leq k,1\leq j\leq 10MC\},
\end{equation*}
 it contains $10kMC>2n_0$ elements. On the other hand, it has an upper bound of $2n_0$, hence there must be some $i\neq j,\ p\neq q$ with $a_{s_p}+a_{n_i}=a_{s_q}+a_{n_j}$. It follows that
 \begin{equation*}
  ||U^{a_{n_i}}_Td-U^{a_{n_j}}_Td||_{L^1}=||U^{a_{s_p}}_Td-U^{a_{s_q}}_Td||_{L^1}\leq17M\epsilon.
 \end{equation*}
 A contradiction.
\end{proof}
Since rigidity implies zero entropy, we have following corollary:
\begin{cor}
	Let $(X,T)$ be a t.d.s.. If there exists a subsequence $A$ of $\N$ with $\bar{D}(A)>0$ such that $(X,T)$ is $A$-mean-equicontinuous, then the topological entropy of $(X,T)$ is zero.
\end{cor}

\subsection{Results with respect to a function}
Let $(X,T)$ be a t.d.s., let $\mu\in M(X,T)$ and $f\in L^2(\mu)$, let $K$ be a subset of $X$ and $A=\{a_i:i\in\N\}$ be a subsequence of $\N$. 

For $x,y\in X$, define
\begin{equation*}
d_f(x,y)=|f(x)-f(y)|
\end{equation*}
and
\begin{equation*}
\bar{d}_{n,f}^A(x,y)=\frac{1}{n}\sum_{i=0}^{n-1}d_f(T^{a_i}x,T^{a_i}y).  
\end{equation*}
Then $d_f$ and $\bar{d}_{n,f}^A$ are pseudo-metrics.

\medskip

We say that $K$ is \textit{$f$-$A$-equicontinuous in the mean} if for any $\epsilon>0$, there exists $\delta>0$ such that 
\begin{equation*}
x,y \in K,d(x,y)<\delta\Rightarrow \bar{d}_{n,f}^A(x,y)<\epsilon,\forall n\in\N.
\end{equation*}

 We say $f$ is  \textit{$A$-equicontinuous in the mean} if $X$ is $f$-$A$-equicontinuous in the mean.

We say $f$ is \textit{$\mu$-$A$-equicontinuous in the mean} if for any $\tau>0$, there exists a compact subset $K$ of $X$ with $\mu(K)>1-\tau$ such that $K$ is $f$-$A$-equicontinuous in the mean.

\begin{prop}\label{pm}
	Let $(X,T)$ be a t.d.s., $\mu\in M(X,T)$ and $f\in L^2(\mu)$. Let $A=\{a_i:i\in\N\}$ be a subsequence of $\N$. Then the following statements are equivalent:
	\begin{enumerate}
		\item $f$ is $\mu$-$A$-equicontinuous in the mean.
		\item For any $\tau>0$ and $\epsilon>0$, there exist a compact subset $K$ of $X$ with $\mu(K)>1-\tau$ and $\delta>0$ such that 
		\begin{equation*}
		x,y\in K,d(x,y)<\delta\Rightarrow\bar{d}_{n,f}^A(x,y)<\epsilon,\ \forall n\in\N.
		\end{equation*}
	\end{enumerate}
\end{prop}
We say that $K$ is \textit{$f$-$A$-mean-equicontinuous} if for any $\epsilon>0$, there exists $\delta>0$ such that 
\begin{equation*}
x,y \in K,d(x,y)<\delta\Rightarrow \limsup_{n\rightarrow\infty}\bar{d}_{n,f}^A(x,y)<\epsilon.
\end{equation*}

 We say that $f$ is  \textit{$A$-mean-equicontinuous} if $X$ is $f$-$A$-mean-equicontinuous.

We say that $f$ is \textit{$\mu$-$A$-mean-equicontinuous } if for any $\tau>0$, there exists a compact subset $K$ of $X$ with $\mu(K)>1-\tau$ such that $K$ is $f$-$A$-mean-equicontinuous.

\medskip

The notions of $f$-$\mu$-$A$-equicontinuous in the mean  and $f$-$\mu$-$A$-mean-equicontinuous are equivalent. Note that the proof of Theorem \ref{eq} is similar to \cite[Theorem 4.3]{h}.
\begin{thm}\label{eq}
	Let $(X,T)$ be a t.d.s., $\mu\in M(X,T)$ and $f\in L^2(\mu)$. Let $A=\{a_i:i\in\N\}$ be a subsequence  of $\N$, then the following statements are equivalent:
	\begin{enumerate}
		\item $f$ is $\mu$-$A$-equicontinuous in the mean.
		\item $f$ is $\mu$-$A$-mean-equicontinuous.
	\end{enumerate}	
\end{thm}
\begin{proof}
	(1)$\Rightarrow$(2) is obvious.
	
	(2)$\Rightarrow$(1): Fix $\tau>0$ and $\epsilon>0$. There exists a compact set $K_0$ with $\mu(K_0)>1-\dfrac{\tau}{2}$ such that $K_0$ is $f$-$A$-mean-equicontinuous. Then there exists $\delta_0>0$ such that for all $x,y\in K_0$ with $d(x,y)<\delta_0$, we have 
	\begin{equation*}
	\limsup_{n\rightarrow\infty}\frac{1}{n}\sum_{i=0}^{n-1}|f(T^{a_i}x)-
	f(T^{a_i}y)|<\frac{\epsilon}{2}.
	\end{equation*}
By the compactness of $X$, there exist $K_0=\bigcup\limits_{i=1}^mX_i$, where $diamX_i<\delta_0,1\leq i\leq m,\ m\in\N.$ Without loss of generality, assume $X_i\neq \emptyset$ and choose $x_i\in X_i,1\leq i\leq m.$ For $1\leq j\leq m$ and $N\in\N$, let 
	\begin{equation*}
	A_N(x_j)=\{y\in X_j:\frac{1}{n}\sum_{i=0}^{n-1}|f(T^{a_i}x_j)-
	f(T^{a_i}y)|<\frac{\epsilon}{2},\forall n>N\}.
	\end{equation*}
	It is easy to see that for each $1\leq j\leq m$, $\{A_N(x_j)\}_{N=1}^\infty$ is increasing and 
	\begin{equation*}
	X_j=\bigcup\limits_{N=1}^\infty A_N(x_j).
	\end{equation*}
	 Choose $N\in\N$ such that 
\begin{equation*}
\mu(\bigcup\limits_{j=1}^mA_N(x_j))>1-\dfrac{\tau}{2}.
\end{equation*}			
Since $f,U_T^{a_1}f,\ldots,U_T^{a_N}f\in L^2(\mu)$, there exists a compact set $K_1$ with $\mu(K_1)>1-\dfrac{\tau}{2}$ such that
$f,U_T^{a_1}f,\ldots,U_T^{a_N}f$ are uniformly continuous on $K_1$. Then there exists $\delta_1>0$ such that for any $x,y\in K_1$ and $d(x,y)<\delta_1$, we have
\begin{equation*}
|f(x)-f(y)|<\epsilon,|f(T^{a_n}x)-
f(T^{a_n}y)|<\epsilon,1\leq n\leq N. 
\end{equation*}
 Note that
 \begin{equation*}
  \mu(K_1\cap\bigcup_{j=1}^mA_N(x_j))>1-\tau,
 \end{equation*} 
by the regularity of $\mu$, we can find  pairwise disjoint compact sets
\begin{equation*}
 E_j\subset A_N(x_j),1\leq j\leq m
\end{equation*} 
such that
\begin{equation*}
\bigcup_{i=j}^mE_j\subset K_1\cap\bigcup_{j=1}^mA_N(x_j),\ \mu(\bigcup_{j=1}^mE_j)>1-\tau.
\end{equation*} 
Let $K=\bigcup\limits_{j=1}^mE_j$, then $K$ is compact and $\mu(K)>1-\tau$. 
Choose $\delta$ with 
\begin{equation*}
0<\delta<\min\{\delta_1,d(E_i,E_j),1\leq i\neq j\leq m\}.
\end{equation*}
 For any $x,y\in K$ and $d(x,y)<\delta$, if $1\leq n\leq N$, since $x,y\in K_1$, we have $\dfrac{1}{n}\sum\limits_{i=0}^{n-1}|f(T^{a_i}x)-
 f(T^{a_i}y)|<\epsilon$; if $n>N$, since $d(x,y)<\delta$, there exists $1\leq j_0\leq m$ such that $x,y\in E_{j_0}$. Then
\begin{equation*}
\frac{1}{n}\sum_{i=0}^{n-1}|f(T^{a_i}x_{j_0})-
f(T^{a_i}x)|<\frac{\epsilon}{2},\
\frac{1}{n}\sum_{i=0}^{n-1}|f(T^{a_i}x_{j_0})-
f(T^{a_i}y)|<\frac{\epsilon}{2}.
\end{equation*}
Hence
$\dfrac{1}{n}\sum\limits_{i=0}^{n-1}|f(T^{a_i}x)-
f(T^{a_i}y)|<\epsilon.$
By Proposition \ref{pm}, $f$ is $\mu$-$A$-equicontinuous in the mean.
\end{proof}

The main theorem of this subsection is:
\begin{thm}\label{fmean}
	Let $(X,T)$ be a t.d.s. $\mu\in M(X,T)$ and $f\in L^2(\mu)$. If there exists a subsequence $A=\{a_i:i\in\N\}$ of $\N$ with $\bar{D}(A)>0$ such that $f$ is $\mu$-$A$-mean-equicontinuous, then $f$ is rigid.
\end{thm}

In order to prove  Theorem \ref{fmean}, we need the following theorem. The idea is from \cite[Theorem 4.4]{h}.
\begin{thm}\label{fm}
	Let $(X,T)$ be a t.d.s. $\mu\in M(X,T)$ and $f\in L^2(\mu)$. If there exists a subsequence $A=\{a_i:i\in\N\}$ of $\N$ such that  $U_T^{a_n}d_f\stackrel{a.e.}\longrightarrow d_f$,  then  $f$ is $\mu$-$A$-equicontinuous.
\end{thm}
\begin{proof}
Given $\tau>0$ and $\epsilon>0$. Since $f\in L^2(\mu)$, there exists a compact set $K_0$ with $\mu(K_0)>1-\dfrac{\tau}{4}$ such that $f$ is uniformly continuous on $K_0$. Then there exists $\delta_0>0$ such that for any $x,y\in K_0$ and $d(x,y)<\delta_0$, we have $d_f(x,y)<\dfrac{\epsilon}{8}$. By the compactness of $X$, there exist $K_0=\bigcup\limits_{i=1}^CX_i$, where $diamX_i<\delta_0,1\leq i\leq C,\ C\in\N$. Without loss of generality, assume $X_i\neq \emptyset$ and choose $x_i\in X_i,\ 1\leq i\leq C.$ It is easy to see that 
$K_0\subset \bigcup\limits_{i=1}^CB_{d_f}(x_i,\dfrac{\epsilon}{8}).$
It follows that
\begin{equation*}
\mu(\bigcup_{i=1}^CB_{d_f}(x_i,\dfrac{\epsilon}{8}))>1-\dfrac{\tau}{4}. 
\end{equation*} 
 
  For any $N\in\N$, let
\begin{equation*}
A_N=\{(x,y):|U_T^{a_n}d_f(x,y)-d_f(x,y)|<\dfrac{\epsilon}{4}, \ \forall n>N\}.
\end{equation*}
It is easy to see that 
\begin{equation*}
\{(x,y):U_T^{a_n}d_f(x,y)\rightarrow d_f(x,y)\}\subset \bigcup_{N=1}^\infty A_N.
\end{equation*}
Note that $\{A_N\}$ is increasing, so there exists $N\in\N$ such that $\mu\times\mu(A_N)>1-(\dfrac{\tau}{4C})^2$. By Fubini's theorem there exists  a subset $D$ of $X$ with $\mu(D)>1-\dfrac{\tau}{4C}$ such that $\forall x\in D,\mu((A_N)_{x})>1-\dfrac{\tau}{4C},$ where $(A_N)_{x}=\{y\in X:(x,y)\in A_N\}$.  

Note that
\begin{equation*}
\mu(\bigcup_{i=1}^C(B_{d_f}(x_i,\dfrac{\epsilon}{8})\cap D))>1-\dfrac{\tau}{2},
\end{equation*} 
 without loss of generality, assume  $B_{d_f}(x_i,\dfrac{\epsilon}{8})\cap D\neq\emptyset$ and choose 
 \begin{equation*}
 y_i\in B_{d_f}(x_i,\dfrac{\epsilon}{8})\cap D,1\leq i\leq C.
 \end{equation*}
  Then 
  \begin{equation*}
  \mu((A_N)_{y_i})>1-\dfrac{\tau}{4C},1\leq i\leq C.
  \end{equation*} 
  Since $B_{d_f}(x_i,\dfrac{\epsilon}{8})\subset B_{d_f}(y_i,\dfrac{\epsilon}{4})$, we have \begin{equation*}
\mu(\bigcup_{i=1}^CB_{d_f}(y_i,\dfrac{\epsilon}{4}))\geq\mu(D\cap\bigcup_{i=1}^CB_{d_f}(y_i,\dfrac{\epsilon}{4}))>1-\dfrac{\tau}{2}.
\end{equation*}
Since $U_T^{a_1}f,\ldots,U_T^{a_N}f\in L^2(\mu)$, there exists a compact set $K_1$ with $\mu(K_1)>1-\dfrac{\tau}{4}$ such that $U_T^{a_1}f,\ldots,U_T^{a_N}f$ are uniformly continuous on $K_1$. Then there exists $\delta_1>0$ such that for any $x,y\in K_1$ and $d(x,y)<\delta_1$, we have
\begin{equation*}
U_T^{a_n}d_f(x,y)<\epsilon,1\leq n\leq N.
\end{equation*} 
Note that
\begin{equation*}
\mu(K_1\cap\bigcup_{i=1}^CB_{d_f}(y_i,\dfrac{\epsilon}{4})\cap\bigcap_{i=1} ^C(A_N)_{y_i})>1-\tau,
 \end{equation*}
  by the regularity of $\mu$, we can find  pairwise disjoint compact sets
  \begin{equation*}
  E_i\subset B_{d_f}(y_i,\dfrac{\epsilon}{4}),1\leq i\leq C
  \end{equation*}  such that
  \begin{equation*}
  \bigcup_{i=1}^CE_i\subset K_1\cap\bigcup_{i=1}^CB_{d_f}(y_i,\dfrac{\epsilon}{4})\cap\bigcap_{i=1} ^C(A_N)_{y_i}
  \end{equation*} 
   and $\mu(\bigcup\limits_{i=1}^CE_i)>1-\tau.$ Let $K=\bigcup\limits_{i=1}^CE_i$, then $K$ is compact and $\mu(K)>1-\tau$. Choose $\delta$ with 
   \begin{equation*}
   0<\delta<\min\{\delta_1,d(E_i,E_j),1\leq i\neq j\leq C\}.
   \end{equation*}
 For any $x,y\in K$ and $d(x,y)<\delta$, if $1\leq n\leq N$, since $x,y\in K_1$, $U_T^{a_n}d_f(x,y)<\epsilon$; if $n>N$, since $d(x,y)<\delta$, there exists $1\leq i_0\leq C$ such that $x,y\in E_{i_0}$. Then
\begin{equation*}
U_T^{a_n}d_f(x,y)\leq U_T^{a_n}d_f(y_{i_0},x)+U_T^{a_n}d_f(y_{i_0},y).
\end{equation*}
Note that $x,y\in (A_N)_{y_{i_0}}$, we have $(y_{i_0},x),(y_{i_0},y)\in A_N$, hence
\begin{equation*}
U_T^{a_n}d_f(y_{i_0},x)<d_f(y_{i_0},x)+\dfrac{\epsilon}{4}.
\end{equation*}
Since $E_{i_0}\subset B_{d_f}(y_{i_0},\dfrac{\epsilon}{4}),$ we have $d_f(y_{i_0},x)<\dfrac{\epsilon}{4},$ it follows that $U_T^{a_n}d_f(y_{i_0},x)<\dfrac{\epsilon}{2}$. Similarly  $U_T^{a_n}d_f(y_{i_0},y)<\dfrac{\epsilon}{2}$, so $U_T^{a_n}d_f(x,y)<\epsilon$. By Proposition \ref{p}, $f$ is $\mu$-$A$-equicontinuous.
	\end{proof}

Now we prove Theorem \ref{fmean}:
\begin{proof}[proof of Theorem \ref{fmean}]
	Replacing $U_T^{a_n}d$ by $U_T^{a_n}d_f$, the proof of Theorem \ref{mean} still works, so we have ${\{U_T^{a_n}d_f:n\in \N\}}$  is pre-compact in $L^1(\mu\times\mu)$. Then there exists a subsequence $S=\{s_i:i\in\N\}$ of $\N$ such that  $U_T^{s_n}d_f\stackrel{a.e.}\longrightarrow s_f$.  By
	Theorem	\ref{fm},  $f$ is $\mu$-$S$-equicontinuous. By Theorem \ref{f}, $f$ is rigid.
	\end{proof} 
\begin{rem}
When considering the measure-theoretic mean equicontinuity along some subsequence $A$ of $\N$, things may get relatively complicated, since some useful tools and methods in ergodic theory do not work. But there are still some interesting questions to consider.
\end{rem}
\begin{ques}
	Is the condition $\bar{D}(A)>0$ in Theorem \ref{mean} necessary?
\end{ques}
\begin{ques}
	If we take $A$ as $\mathbb{P}$, the set of all prime numbers, can we obtain some interesting results?
\end{ques}

\bibliographystyle{plain}
\bibliography{measure-theoretic_equicontinuity_and_rigidity}

\begin{thebibliography}{10}

\bibitem{me1}
Joseph Auslander.
\newblock Mean-{$L$}-stable systems.
\newblock {\em Illinois J. Math.}, 3:566--579, 1959.

\bibitem{ra}
V.~Bergelson, A.~del Junco, M.~Lema\'{n}czyk, and J.~Rosenblatt.
\newblock Rigidity and non-recurrence along sequences.
\newblock {\em Ergodic Theory Dynam. Systems}, 34(5):1464--1502, 2014.

\bibitem{ri}
Hillel Furstenberg and Benjamin Weiss.
\newblock The finite multipliers of infinite ergodic transformations.
\newblock In {\em The structure of attractors in dynamical systems ({P}roc.
  {C}onf., {N}orth {D}akota {S}tate {U}niv., {F}argo, {N}.{D}., 1977)}, volume
  668 of {\em Lecture Notes in Math.}, pages 127--132. Springer, Berlin, 1978.

\bibitem{fe1}
Felipe Garc\'{i}a-Ramos.
\newblock A characterization of {$\mu$}-equicontinuity for topological
  dynamical systems.
\newblock {\em Proc. Amer. Math. Soc.}, 145(8):3357--3368, 2017.

\bibitem{fe2}
Felipe Garc\'{i}a-Ramos.
\newblock Weak forms of topological and measure-theoretical equicontinuity:
  relationships with discrete spectrum and sequence entropy.
\newblock {\em Ergodic Theory Dynam. Systems}, 37(4):1211--1237, 2017.

\bibitem{f1}
Felipe {Garc{\'\i}a-Ramos} and Brian {Marcus}.
\newblock {Mean sensitive, mean equicontinuous and almost periodic functions
  for dynamical systems}.
\newblock {\em arXiv e-prints}, page arXiv:1509.05246, Sep 2015.

\bibitem{Gil1}
Robert~H. Gilman.
\newblock Classes of linear automata.
\newblock {\em Ergodic Theory Dynam. Systems}, 7(1):105--118, 1987.

\bibitem{Gil2}
Robert~H. Gilman.
\newblock Periodic behavior of linear automata.
\newblock In {\em Dynamical systems ({C}ollege {P}ark, {MD}, 1986--87)}, volume
  1342 of {\em Lecture Notes in Math.}, pages 216--219. Springer, Berlin, 1988.

\bibitem{ur}
S.~Glasner and D.~Maon.
\newblock Rigidity in topological dynamics.
\newblock {\em Ergodic Theory Dynam. Systems}, 9(2):309--320, 1989.

\bibitem{h}
Wen {Huang}, Jian {Li}, Jean-Paul {Thouvenot}, Leiye {Xu}, and Xiangdong {Ye}.
\newblock {Bounded complexity, mean equicontinuity and discrete spectrum}.
\newblock {\em arXiv e-prints}, page arXiv:1806.02980, June 2018.

\bibitem{hly}
Wen Huang, Ping Lu, and Xiangdong Ye.
\newblock Measure-theoretical sensitivity and equicontinuity.
\newblock {\em Israel J. Math.}, 183:233--283, 2011.

\bibitem{hsy}
Wen Huang, Song Shao, and Xiangdong Ye.
\newblock Mixing via sequence entropy.
\newblock In {\em Algebraic and topological dynamics}, volume 385 of {\em
  Contemp. Math.}, pages 101--122. Amer. Math. Soc., Providence, RI, 2005.

\bibitem{hy1}
Wen Huang and Xiangdong Ye.
\newblock Topological complexity, return times and weak disjointness.
\newblock {\em Ergodic Theory Dynam. Systems}, 24(3):825--846, 2004.

\bibitem{hyz}
Wen Huang, Xiangdong Ye, and Guohua Zhang.
\newblock A local variational principle for conditional entropy.
\newblock {\em Ergodic Theory Dynam. Systems}, 26(1):219--245, 2006.

\bibitem{li}
Jian Li, Siming Tu, and Xiangdong Ye.
\newblock Mean equicontinuity and mean sensitivity.
\newblock {\em Ergodic Theory Dynam. Systems}, 35(8):2587--2612, 2015.

\bibitem{me2}
Jiahao {Qiu} and Jianjie {Zhao}.
\newblock {A note on mean equicontinuity}.
\newblock {\em arXiv e-prints}, page arXiv:1806.09987, Jun 2018.

\bibitem{metric}
Anatoly~M. Vershik, Pavel~B. Zatitskiy, and Fedor~V. Petrov.
\newblock Geometry and dynamics of admissible metrics in measure spaces.
\newblock {\em Cent. Eur. J. Math.}, 11(3):379--400, 2013.

\bibitem{pw}
Peter Walters.
\newblock {\em An introduction to ergodic theory}, volume~79 of {\em Graduate
  Texts in Mathematics}.
\newblock Springer-Verlag, New York-Berlin, 1982.

\bibitem{f2}
Tao {Yu}.
\newblock {Measure-theoretic mean equicontinuity and bounded complexity}.
\newblock {\em arXiv e-prints}, page arXiv:1807.05868, Jul 2018.

\end{thebibliography}

\address{Wu Wen-Tsun Key Laboratory of Mathematics, USTC, Chinese Academy of Sciences and
	Department of Mathematics, University of Science and Technology of China,
	Hefei, Anhui, 230026, P.R. China.}

\email{cfz@mail.ustc.edu.cn}

\end{document}